\documentclass[12pt]{amsart}

\usepackage{tikz}
\usepackage{amsmath, amsfonts, amssymb}

\usetikzlibrary{shapes.geometric,calc}
\usepackage{bm}
\usepackage{hyperref}
\usepackage{circledsteps}

\usepackage[margin=1in,letterpaper,portrait]{geometry}

\theoremstyle{definition}
\newtheorem{definition}{Definition}[section]
\newtheorem*{remark}{Remark}
\newtheorem{example}[definition]{Example}
\newtheorem{obs}[definition]{Observation}

\theoremstyle{plain}
\newtheorem{theorem}{Theorem}
\newtheorem{proposition}[definition]{Proposition}
\newtheorem{lemma}[definition]{Lemma}
\newtheorem{corollary}[definition]{Corollary}

\DeclareMathOperator{\h}{dr}
\DeclareMathOperator{\pr}{Pr}
\newcommand{\perms}{\widehat{S}_n}%% new
\newcommand{\alg}{\mathcal{C}}%% new
\newcommand{\lb}{\Big(}
\newcommand{\rb}{\Big)}

\begin{document}

\title{The clairvoyant ma\^{i}tre d'}

\author[Acton]{Reed Acton}
\address{Department of Mathematical Sciences, DePaul University, Chicago, IL, USA}
\email{racton1@depaul.edu}

\author[Petersen]{T.~Kyle Petersen}
\address{Department of Mathematical Sciences, DePaul University, Chicago, IL, USA}
\email{t.kyle.petersen@depaul.edu}

\author[Shirman]{Blake Shirman}
\address{Department of Mathematical Sciences, DePaul University, Chicago, IL, USA}
\email{blake.kohrmann@gmail.com}

\author[Tenner]{Bridget Eileen Tenner}
\address{Department of Mathematical Sciences, DePaul University, Chicago, IL, USA}
\email{bridget@math.depaul.edu}

%\author[Toal]{Daniel Toal}
%\address{Department of Mathematical Sciences, DePaul University, Chicago, IL, USA}
%\email{dtoal@depaul.edu}

\maketitle

\begin{abstract}
In this paper we study a variant of the Malicious Ma\^{i}tre d' problem. This problem, attributed to computer scientist Rob Pike in Peter Winkler's book \emph{Mathematical Puzzles: A Connoisseur's Collection}, involves seating diners around a circular table with napkins placed between each pair of adjacent settings. The goal of the ma\^{i}tre d' is to seat the diners in a way that maximizes the number of diners who arrive at the table to find the napkins on both the left and right of their place already taken by their neighbors. Previous work described a seating algorithm in which the ma\^{i}tre d' expects to force about 18\% of the diners to be napkinless. In this paper, we show that if the ma\^{i}tre d' learns each diner's preference for the right or left napkin before they are placed at the table, this expectation jumps to nearly $1/3$ (and converges to $1/3$ as the table size gets large). Moreover, our strategy is optimal for every sequence of diners' preferences. 
\end{abstract}

\section{The clairvoyant ma\^{i}tre d'.}

This work is a follow-up to recent work by four of the authors \cite{APST}. That prior work was motivated by a problem from \textit{Mathematical Puzzles: A Connoisseur's Collection}, by Peter Winkler \cite[p.~22]{Winkler}. In that problem, titled ``The Malicious Ma\^{i}tre d','' there is a circular table with $n$ chairs and $n$ napkins---one napkin placed between each pair of consecutive chairs. A group of $n$ people is seated around the table, taking napkins to the left or the right of their seats---by random choice when there is a choice, by necessity when there is no choice---and the ma\^{i}tre d' is trying to seat the diners in such an order so as to maximize the number of diners who find no napkin available for choosing. Such diners are \emph{napkinless}. 

Winkler's book provides solutions to the Malicious Ma\^{i}tre d' problem and a variant in which the diners sit randomly. With random seating, roughly $12\%$ of the diners are napkinless, on average, while a maliciously chosen seating arrangement can force about $14\%$ of diners to be napkinless. (Further study of the random seating problem, also known as ``Conway's napkin problem,'' can be found in \cite{CP, Eriksen, Sudbury}.) In his solution to the original problem, Winkler suggests another version of the puzzle, in which the ma\^{i}tre d' observes which napkin a diner selects after they sit down. In \cite{APST} this is called ``The Adaptive Ma\^{i}tre d','' since although the ma\^{i}tre d' has no foreknowledge of the napkin a diner will prefer, they can wait until after making the observation before they decide where to place subsequent diners. The focus of \cite{APST} was a comparison of two strategies for the adaptive ma\^{i}tre d' puzzle: the \emph{trap setting strategy}, which produces a proportion of about $1/6$ napkinless diners, and the \emph{napkin shunning strategy}, which produces close to $18\%$ napkinless diners.

The thrust of the present paper is to modify the assumptions once more. The ma\^{i}tre d' is not just adaptive, but now they are also clairvoyant. That is, while the newest diner and the ma\^{i}tre d' walk to the table to be seated, the ma\^{i}tre d' receives a vision (if the ma\^{i}tre d' does not possess supernatural powers, they might just ask the diner's preference) that reveals whether this diner will prefer to choose the left napkin or the right napkin. We will refer to this new puzzle as ``The Clairvoyant Ma\^{i}tre d','' and the rest of this paper is devoted to this puzzle.  As in \cite{APST}, preference for either the right or left napkin is assumed to be equal. 

Our main achievement is a strategy that we call \emph{clairvoyant trap setting}, for which we prove the following results.

\begin{theorem}\label{thm:main}
Let $n\geq 3$ and $q=\lfloor n/3 \rfloor$. Under the clairvoyant trap setting strategy, we have the following results.
\begin{enumerate}
\item The probability of $k$ napkinless diners at a table with $n$ seats is $p_{n,k}/2^n$, where:
\[
p_{n,k}=
 \begin{cases}
  4\displaystyle\binom{n}{k} & \mbox{ if $0\leq k < q$,} \\
  2^n-4\displaystyle\sum_{i=0}^{q-1}\binom{n}{i} & \mbox{ if $k=q$,}\\
  0 &\mbox{ otherwise.}
 \end{cases}
\] 
\item The expected number of napkinless diners is
\[
 E_n = q - \frac{1}{2^{n-2}}\sum_{k=0}^{q-1} (q-k)\binom{n}{k},
\]
and $E_n \to q =\lfloor n/3 \rfloor$ as $n\to \infty$.

\item Clairvoyant trap setting is optimal in the sense that, for any sequence of napkin preferences, the strategy maximizes the number of napkinless diners.

\end{enumerate}
\end{theorem}

Figure~\ref{fig: plots} illustrates, for $3\leq n\leq 100$, the differences in the proportions of napkinless diners for the clairvoyant trap setting strategy versus the non-clairvoyant strategies studied in \cite{APST}.

\begin{figure}
\includegraphics[width=14cm]{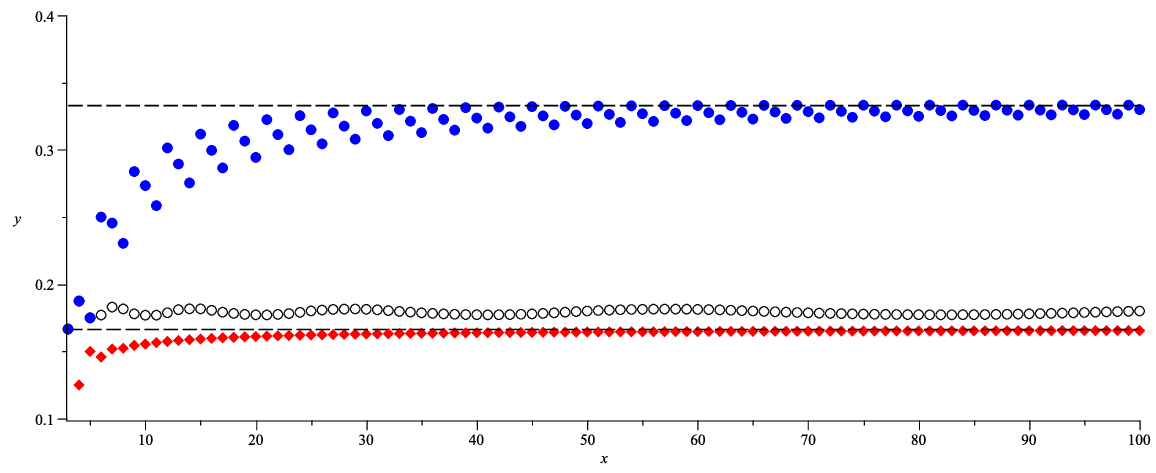}
\caption{The expected proportion of napkinless diners under the trap setting strategy (diamonds), the napkin shunning strategy (open circles), and the clairvoyant trap setting strategy (filled circles). The dashed lines are at heights $1/6$ and $1/3$.}\label{fig: plots}
\end{figure}

In Section~\ref{sec:notation} we establish our notation for the problem so that we can translate the space of potential strategies for the ma\^{i}tre d' into a set of combinatorial objects, and we present preliminary results in Section~\ref{sec:preliminary results}. Having characterized the problem combinatorially, we proceed to identify bounds on the number of napkinless diners, for any strategy, in Section~\ref{sec:bounding}. We study our clairvoyant trap setting algorithm and prove its optimality (Theorem~\ref{thm:main} part (3)) in Section~\ref{sec:clairvoyant}. The probabilistic results in parts (1) and (2) of Theorem~\ref{thm:main} are proved in Section~\ref{sec:distribution}. A key tool in our analysis is a statistic for lattice paths that we call \emph{drift}. This statistic may be of independent interest, and in Section~\ref{sec:drift} we show the number of lattice paths of length $n$ and drift $h$ is the binomial coefficient 
$$\binom{n}{\lfloor (n-h)/2 \rfloor}.$$

\section{Notation and terminology}\label{sec:notation}

Assume throughout this work that there are $n$ diners (and $n$ chairs and $n$ napkins).

The setup of this problem begins with a queue of diners, each of whom has an inherent left/right preference for which napkin they would take when given a choice. The clairvoyant ma\^{i}tre d' learns each diner's preference as they arrive at the front of the queue. We model a \emph{preference order} for the diners by a list
$$\sigma = (\sigma_1, \sigma_2, \ldots, \sigma_n) \in \{\pm1\}^n,$$
where $\sigma_j = -1$ means that the $j$th diner prefers the napkin to their left and $\sigma_j = +1$ means the diner prefers the napkin to their right. We will refer to a diner preferring their left napkin as a \emph{negative} diner, and a diner preferring their right napkin as a \emph{positive} diner. We reiterate that the ma\^{i}tre d' only learns $\sigma_i$ after the first $i-1$ diners have been seated.

We label the diners by their order of arrival at the table (equivalently, by their initial position in the queue to be seated). The dining table is circular, so we label the seats by their counterclockwise position relative to the seat of Diner 1. That is, Diner 1 sits in Seat 1, Seat 2 is to their right, and Seat $n$ is to their left. It will be understood that seat indices are always taken modulo $n$. We write \emph{seating orders} as permutations $w=(w_1,w_2,\ldots,w_n)$, where $w_i = j$ means that Diner $j$ sits in Seat $i$; note that we always have $w_1 = 1$. We use the notation $\perms$ for the set of all such permutations. 

\begin{definition}\label{defn:seating arrangement}
A \emph{seating arrangement} $(w,\sigma)$ for a group of diners consists of a preference order $\sigma$ and a seating order $w$.
We can simultaneously record these pieces of data by writing 
\begin{equation}\label{eqn:seating order notation}
(\sigma'_1 w_1, \sigma'_2 w_2, \ldots, \sigma'_n w_n),
\end{equation}
where $\sigma'_i = \sigma_j$ when $w_i = j$.
\end{definition}

Certainly both $w$ and $\sigma$ can be recovered from Expression~\eqref{eqn:seating order notation}.

For a given seating arrangement, some diners will be able to select their preferred napkins, others will have to take a non-preferred napkin, and others will find no napkins available to them. The people in this last group are the targets of our attention.

\begin{definition}\label{defn:napkinless}
Fix a seating arrangement $(w,\sigma)$. A diner who is seated by the ma\^{i}tre d' in such a way as to find both their nearest left and right napkins already claimed is \emph{napkinless}.
\end{definition}

\begin{center}
\framebox{
\begin{minipage}{4.1in}
\vspace{.1in}
The goal of our clairvoyant ma\^{i}tre d' is to find a seating order so as to maximize the number of napkinless diners.
 \vspace{.05in}
\end{minipage}}
\end{center}

In principle, any permutation $w \in \perms$ could be a seating order used by the ma\^{i}tre d', and any seating strategy used by the ma\^{i}tre d' will result in a seating order that is a permutation. Thus to characterize the outcomes of any strategy is to characterize the seating orders it produces, and claims of optimality for a given strategy can be checked against the outcomes for all seating orders $w\in \perms$.

\begin{example}\label{ex:8 diners}
Suppose there are eight diners, with preference order
$$\sigma = (1, -1, -1, 1, 1, -1, 1, -1),$$
and the ma\^{i}tre d' gives them the seating order 
$$w = (1,5,2,8,4,6,7,3).$$
We can compute the vector $\sigma' = (1, 1, -1, -1, 1, -1, 1, -1)$, and we can represent this entire scenario with the seating arrangement
$$(1,\Circled{5}, -2, -8, 4, -6, \Circled{7}, -3).$$
Equivalently, although less compactly, Figure~\ref{fig:circle} shows how the diners would select their napkins under these circumstances. Diners 5 and 7 (circled in the arrangement above, and sitting in Seats 2 and 7, respectively) arrive at the table to find no napkins available, and hence they are napkinless.
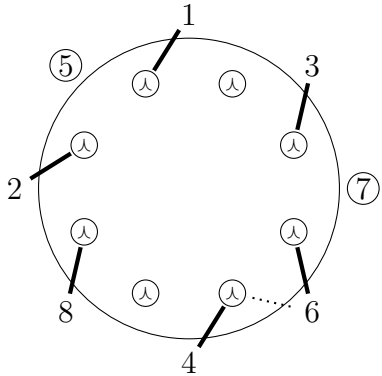
\begin{figure}[htbp]
\[
\begin{tikzpicture}[scale=1,baseline=0]
\draw (0,0) circle (2);
x\node[draw=none,minimum size=5cm,regular polygon,regular polygon sides=8] (a) {};
\node[draw=none,minimum size=3cm,regular polygon,regular polygon sides=8] (c) {};
\foreach \x in {1,2,3,4,5,6,7,8}
  \draw (c.corner \x) node {$\curlywedge$};
\draw[ultra thick] (a.side 1) -- (c.corner 2);
\draw[ultra thick] (a.side 3) -- (c.corner 3);
\draw[ultra thick] (a.side 4) -- (c.corner 4);
\draw[ultra thick] (a.side 5) -- (c.corner 6);
\draw[ultra thick] (a.side 6) -- (c.corner 7);
\draw[ultra thick] (a.side 8) -- (c.corner 8);
\draw[thick, dotted] (a.side 6) -- (c.corner 6);
\foreach \x in {1,2,3,4,5,6,7,8}
  {\fill[white] (c.corner \x) circle (6pt);
  \draw (c.corner \x) circle (5pt);
  \draw (c.corner \x) node {$\scriptstyle{\curlywedge}$};
  \fill[white] (a.side \x) circle (7pt);
  }
\draw (a.side 1) node {$1$};
\draw (a.side 2) node {$5$};
\draw (a.side 3) node {$2$};
\draw (a.side 4) node {$8$};
\draw (a.side 5) node {$4$};
\draw (a.side 6) node {$6$};
\draw (a.side 7) node {$7$};
\draw (a.side 8) node {$3$};
\foreach \x in {2,7} {\draw (a.side \x) circle (6pt);}
\end{tikzpicture}\]
\caption{The seating arrangement $(1,\Circled{5}, -2, -8, 4, -6, \Circled{7}, -3)$, with preferences that lead to two (circled) napkinless diners (and two unclaimed napkins). The thick lines indicate the napkins claimed by Diners $1$, $2$, $3$, $4$, $6$, and $8$, while the dotted line indicates that Diner $6$ had been a negative diner (wanting the napkin on their left), but was forced to take the napkin on their right.}\label{fig:circle}
\end{figure}
\end{example}

\section{Preliminary results}\label{sec:preliminary results}%\subsection{Location of napkinless diners}

For a particular seating arrangement $(w,\sigma)$, we can count the napkinless diners produced by $(w,\sigma)$. We let
$$\nu(w,\sigma)$$
denote the number of napkinless diners for the seating arrangement $(w,\sigma)$. For the seating arrangement of Example~\ref{ex:8 diners}, we have $\nu(w,\sigma) = 2$.

As we study $\nu(w,\sigma)$, we begin with the following small result.

\begin{lemma}\label{lem:can achieve nu = 0}
For any preference order $\sigma$, there exists a seating order $w$ for which there are no napkinless diners; that is, for which $\nu(w,\sigma) = 0$. 
\end{lemma}

\begin{proof}
If $\sigma_1 = +1$, then $\nu((1,2,\ldots,n),\sigma) = 0$. This is because, by seating the diners sequentially to the right from the first diner, they will each always claim the napkin to their right and no one will be napkinless. If, on the other hand, $\sigma_1 = -1$, then $\nu((1,n,n-1,\ldots, 2),\sigma) = 0$, by a symmetric argument.
\end{proof}

Lemma~\ref{lem:can achieve nu = 0} is encouraging for the diners, but remember: our ma\^{i}tre d' is malicious! They want to find a seating order that creates as many napkinless diners as possible, given their preference order $\sigma$. We write
\[
 \nu_{\max}(\sigma) := \max_{w \in \perms}\{ \nu(w,\sigma)\}.
\]
We can immediately give bounds on $\nu_{\max}(\sigma)$. 

\begin{proposition}\label{prop:bounding numax}
For any preference order $\sigma$ of $n$ diners, we have $0 \leq \nu_{\max}(\sigma) \leq \lfloor n/3 \rfloor$.
\end{proposition}

\begin{proof}
The lower bound follows from Lemma~\ref{lem:can achieve nu = 0}. The upper bound follows because napkinless diners are in bijection with unused napkins, and among any three consecutive napkins, at most one of them is unused. Indeed, if a napkin is unused, that is because the diners adjacent to it have selected the napkins to its left and right.
\end{proof}

These bounds are sharp.  Indeed, $\nu_{\max}(1,1,1,\ldots)=0,$ since every diner will arrive to find their rightward napkin available, regardless of the seating order. On the other hand, $\nu_{\max}(1,-1,1,1,-1,1,\ldots) = \lfloor n/3 \rfloor$, since the seating order $(1,3,2,4,6,5,\ldots)$ will yield napkinless diners in the Seats $\{2, 5, 8, 11, \ldots, 3\lfloor n/3 \rfloor - 1\}$.

As we consider the necessary circumstances for a diner to be napkinless, other diners are useful for reference. 

\begin{definition}\label{defn:happy and frustrated}
Following the language of \cite{CP, Eriksen}, we say that a diner is \emph{happy} if they receive the napkin they prefer, and \emph{frustrated} if they receive a napkin, but not their preferred napkin.
\end{definition}

In Example~\ref{ex:8 diners}, Diner $6$ is frustrated.

Suppose Diner $j$ is napkinless. As observed in \cite{CP,Eriksen}, there is necessarily a nearest happy positive diner to the left of Diner $j$ who arrived prior to Diner $j$, say Diner $j'$ with $j'<j$. 
Because Diner $j'$ is the nearest happy diner to the left of Diner $j$, any diners seated physically between them are frustrated because they hoped for their left napkin but were forced to choose the right napkin. In the same fashion, there is a nearest happy negative Diner $j''$ to the right of Diner $j$. Moreover, the frustrated diners must arrive from the ``outside in'' as illustrated in Figure~\ref{fig:frustrated}.

\begin{figure}[htbp]
%\[
%\begin{tikzpicture}[scale=1,baseline=0]
%\draw (0,0) circle (2);
%x\node[draw=none,minimum size=5cm,regular polygon,regular polygon sides=8] (a) {};
%\node[draw=none,minimum size=3cm,regular polygon,regular polygon sides=8] (c) {};
%\draw[ultra thick] (c.side 3) -- (a.corner 3);
%\draw[ultra thick] (c.side 4) -- (a.corner 4);
%\draw[ultra thick] (c.side 5) -- (a.corner 5);
%\draw[ultra thick] (c.side 6) -- (a.corner 7);
%\draw[ultra thick] (c.side 7) -- (a.corner 8);
%\draw[thick, dotted] (c.side 3) -- (a.corner 4);
%\draw[thick, dotted] (c.side 4) -- (a.corner 5);
%\draw[thick, dotted] (c.side 7) -- (a.corner 7);
%\foreach \x in {2,3,4,5,6,7,8}
%  {\fill[white] (c.side \x) circle (6pt);
%  \draw (c.side \x) circle (5pt);
%  \draw (c.side \x) node {$\scriptstyle{\curlywedge}$};
%  \fill[white] (a.corner \x) circle (7pt);
%  }
%\draw (a.side 1) node[yshift=-10pt, fill=white, inner sep=10] {$\cdots$};
%%\draw (c.side 1) node {$\cdots$};
%\draw (a.corner 3) node {$2$};
%\draw (a.corner 4) node {$3$};
%\draw (a.corner 5) node {$8$};
%\draw (a.corner 6) node {$9$};
%\draw (a.corner 7) node {$7$};
%\draw (a.corner 8) node {$5$};
%\foreach \x in {6} {\draw (a.corner \x) circle (6pt);}
%%%
%\end{tikzpicture}\]
\[
\begin{tikzpicture}[xscale=1.5]
\draw (0.5,.5)--(7.5,.5);
\draw[ultra thick] (1.5,0)--(2,1);
\draw[ultra thick] (2.5,0)--(3,1);
\draw[ultra thick] (3.5,0)--(4,1);
\draw[ultra thick] (5.5,0)--(5,1);
\draw[ultra thick] (6.5,0)--(6,1);
\draw[thick, dotted] (2.5,0)--(2,1);
\draw[thick, dotted] (3.5,0)--(3,1);
\draw[thick, dotted] (5.5,0)--(6,1);
\draw (1.5,0) node[fill=white,circle,inner sep=2] {$2$};
\draw (2.5,0) node[fill=white,circle,inner sep=2] {$3$};
\draw (3.5,0) node[fill=white,circle,inner sep=2] {$8$};
\draw (4.5,0) node[fill=white,draw=black,circle,inner sep=2] {$9$};
\draw (5.5,0) node[fill=white,circle,inner sep=2] {$7$};
\draw (6.5,0) node[fill=white,circle,inner sep=2] {$5$};
\foreach \x in {1,...,7}{
\draw (\x,1) node[fill=white,draw=black,circle,inner sep=2] {$\scriptstyle{\curlywedge}$};
  }
\end{tikzpicture}
\]
\caption{A portion of a seating arrangement in which Diner 9 is napkinless and Diners 2 and 5 are happy. Diners 3, 8, and 7 are frustrated. The minimal napkinless block for Diner 9 is $B(9) = \{2,9,5\}$.}\label{fig:frustrated}
\end{figure}
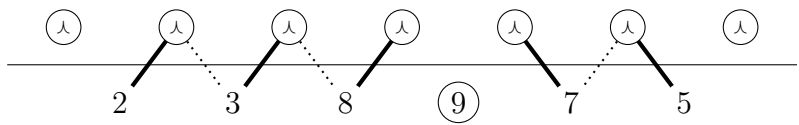

Let us call the set $B(j):= \{j',j,j''\}$ the \emph{minimal napkinless block} for Diner $j$, relative to the seating arrangement $(w,\sigma)$. Notice that if we sit these three diners next to each other at the table, with no other diners physically between them, then Diner $j$ will still be napkinless, and Diners $j'$ and $j''$ will still be happy.

A key insight builds on this notion of a minimal napkinless block to show we can always assume that our napkinless diners are seated as near to each other around the table as possible.

\begin{lemma}\label{lem:knap}
Fix a preference order $\sigma$, with $k:=\nu_{\max}(\sigma)$. There exists a seating order $w \in \perms$ such that the diners in Seats $\{2,5,\ldots, 3k-1\}$ are napkinless. 
\end{lemma}

\begin{proof}
Suppose the diners are seated in such a way that $k$ of them are napkinless. Let $u \in \perms$ denote this seating order.

Label the napkinless Diners $j_1 < j_2 < \cdots < j_k$, and identify their minimal napkinless blocks relative to the seating arrangement $(u,\sigma)$. Let $w$ be the seating order
\[
w=(j_1', j_1, j_1'', j_2', j_2, j_2'',\ldots),
\] 
with all diners not in any $B(j_i)$ included at the end of the seating order (in Seats $3k+1, 3k+2, \ldots, n$). By construction, the seating arrangement $(w,\sigma)$ has napkinless diners in Seats $2, 5, \ldots, 3k-1$. There can be no other napkinless diners because $k$ was maximal.
\end{proof}

We have seen that seating orders with the maximal number of napkinless diners can be constructed from a collection of triples of diners: each consisting of a napkinless diner and the two diners who had already taken the napkinless diner's options. This brings to mind the idea of seating diners on 3-person benches around the table, and so we make the following definition. Note that bench collections, defined below, are independent of any preference order.

\begin{definition}\label{defn:bench collection}
Let $n=3q+r$, with $q = \lfloor n/3\rfloor$. A \emph{bench collection} $\beta = (B_1,\ldots,B_q)$ is a sequence of $q$ disjoint triples (``benches'') such that $B_i = \{a_i < b_i < c_i\} \subseteq [1,n]$ for each $i$. 
\end{definition}

Given a preference order $\sigma \in \{\pm1\}^n$, a bench $B=\{a<b<c\}$ is \emph{balanced} if $\sigma_a+\sigma_b=0$ (i.e., if its two earliest-seated diners have different napkin preferences) and \emph{unbalanced} otherwise. Note that in a balanced bench, we can position those two earliest-seated diners so that the third diner will be seated physically between them -- and with no available napkin. The \emph{balance number} of a bench collection $\beta$, relative to a preference order $\sigma$, is the number of balanced benches that it contains. We denote this by 
\[
 b(\beta,\sigma) := |\{ 1\leq i \leq q : \sigma_{a_i} + \sigma_{b_i} = 0\}|.
\]

From a bench collection $\beta$ and a preference order $\sigma$, we will define a particular seating order.  

\begin{definition}\label{defn:bench seating order}
Fix a preference order $\sigma$ and a bench collection $\beta = (B_1, B_2,\ldots)$. For each bench $B = \{a < b < c\}$, we order its seating via:
$$\widetilde{B} := \begin{cases}
(a, c, b) & \text{if } \sigma_{a} = +1, \text{ and}\\
(b, c, a) & \text{if } \sigma_{a} = -1.
\end{cases}$$
Let $C_1,\ldots, C_{b(\beta,\sigma)}$ be $\{\widetilde{B}_i : B_i \text{ is balanced}\}$, listed with their first components in increasing order, and let $C_{b(\beta,\sigma)+1},\ldots,C_q$ be the ordered unbalanced benches, listed with their first components in increasing order. Let $D$ be an increasing list of any values not included in any bench. Let $v(\beta,\sigma)$ be the permutation defined by the list
$$C_1, C_2, \ldots, C_q, D.$$
This is an element of $S_n$, but it might not have put Diner $1$ into Seat $1$. Define $w(\beta,\sigma) \in \perms$ to be obtained from $v(\beta,\sigma)$ by cycling the positions of the letters of $v(\beta,\sigma)$ until Diner $1$ appears in the first position. 
\end{definition}

Note that the seating order $w(\beta,\sigma)$, determined entirely by $\beta$ and $\sigma$, is naturally associated to the preference order $\sigma$, and so it defines a fixed ``bench seating arrangement'' that we will call
$$(\beta,\sigma) := (w(\beta,\sigma),\sigma).$$

Intuitively, we imagine the seats at the table coming in 3-person benches, possibly with a remainder chair or two. The seating order $w(\beta,\sigma)$ has the property that both ends of a bench are occupied before the middle person sits, and the first person to arrive at the bench is a happy diner who reaches toward the middle of the bench. Moreover, each balanced bench of $\beta$ corresponds to a minimal napkinless block of $(\beta,\sigma)$.
 
\begin{example}\label{ex:14 diners}
Consider $14$ diners with preference order 
$$\sigma = (1,1,-1,1,-1,-1,1,1,1,1,1,1,1,-1).$$
Suppose we have the bench collection
$$\beta = (B_1, B_2, B_3, B_4) = ( \{1, 10, 11\} , \{5, 8, 14\} , \{4,7, 9\}, \{2,6,12\}).$$
Benches $B_2$ and $B_4$ are balanced (because $\sigma_5 + \sigma_8 = 0$ and $\sigma_2 + \sigma_6 = 0$), while benches $B_1$ and $B_3$ are unbalanced, so $b(\beta,\sigma) = 2$. Then 
$$v(\beta,\sigma) = \left(2, 12, 6 \Big| 8, 14, 5 \Big| 1, 11, 10 \Big| 4, 9, 7 \Big| 3, 13\right),$$
where we have used vertical bars to separate the benches. The seating arrangement $(\beta,\sigma)$, obtained by cycling the letters of $v(\beta,\sigma)$ until Diner $1$ appeared in Seat $1$, is
$$\left(1, 11, 10 \Big| 4, \Circled{9}, 7 \Big| -\!3, 13 \Big| 2, \Circled{12}, -6, \Big| 8, -\Circled{14}, -5\right).$$
This arrangement causes Diner $7$ to be frustrated, while Diners $9$, $12$, and $14$ (circled in the list) are all napkinless. Hence $\nu(\beta,\sigma) = 3$.
\end{example}
 
By construction, for any preference order $\sigma$ and any bench collection $\beta$ the number of napkinless diners in the seating arrangement $(\beta, \sigma)$ is at least the number of balanced benches:
\[
 \nu(w,\sigma) \ge b(\beta,\sigma).
\]
Conversely, if a seating order $w$ achieves the maximal number of napkinless diners (that is, if $\nu_{\max}(\sigma) = \nu(w,\sigma)$), then, as exhibited in the proof of Lemma~\ref{lem:knap}, there is a bench collection with $\nu_{\max}(\sigma)$ balanced benches. Taken together, we can characterize $\nu_{\max}(\sigma)$ in terms of bench collections.

\begin{obs}[Bench characterization of napkinless number]\label{obs:balance}
Fix $\sigma \in \{\pm1\}^n$. Then
\[
 \nu_{\max}(\sigma) = \max_{\beta \text{ a bench}\atop \text{collection}}\{ b(\beta,\sigma)\}.
\]
\end{obs}

In other words, we can focus our efforts on maximizing the number of balanced benches in a bench collection. 

\section{Maximizing balanced benches}\label{sec:bounding}

Our first result provides a lower bound on unbalanced benches, phrased in terms of partial sums of a preference order $\sigma$. Before stating the result we introduce some helpful terminology.

First of all, we visualize a preference order $\sigma \in \{\pm1\}^n$ as an $\{N,E\}$ lattice path, by replacing each $+1$ with an ``$N$'' step: $(i,j)\to (i,j+1)$, and each $-1$ with an ``$E$'' step: $(i,j)\to (i+1,j)$. Beginning with $(x_0,y_0)=(0,0)$, a path $p=p(\sigma)$ is a sequence of lattice points: 
\[
p=((x_0,y_0), (x_1,y_1), (x_2,y_2),\ldots, (x_n,y_n)).
\]
With this correspondence, after step $s$, the point $(x_s,y_s)$ has 
\[
x_s=|\{1\leq i\leq s: \sigma_i = -1\}| \quad \mbox{ and } \quad y_s = |\{1\leq i\leq s: \sigma_i = 1\}|,
\]
and $x_s + y_s = s$. We define the \emph{drift} of a path to be 
\[
\h(\sigma) := \max\left( \{0\} \cup \left\{ \sum_{j=1}^i \sigma_j : 1\leq i \leq n\right\}\right) = \max\{ y_i-x_i : 0\leq i \leq n\}.
\]
In other words, $\h(\sigma) = h$ is the maximal $h$ such that the line $y=x+h$ hits a point on the path $p$. 

\begin{example}\label{ex:14 diners path}
Consider the preference order $\sigma = (1,1,-1,1,-1,-1,1,1,1,1,1,1,1,-1)$ from Example~\ref{ex:8 diners}. This corresponds to the path $p(\sigma)$ shown in Figure~\ref{fig:8 diners path}, and the drift $\h(\sigma) = 7$ is achieved at the point $(x_{13},y_{13}) = (3,10)$. 
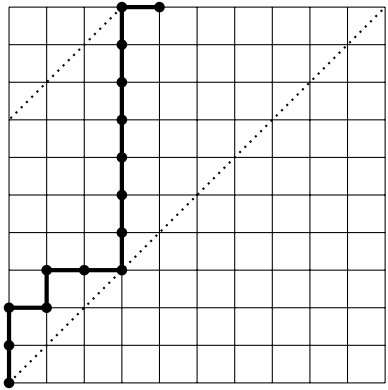
\begin{figure}[htbp]
\begin{tikzpicture}[scale=.5]
\draw (0,0) grid (10,10);
\draw[thick,dotted] (0,0) -- (10,10);
\draw[thick,dotted] (3,10) -- (0,7);
\draw[ultra thick] (0,0) -- (0,2) -- (1,2) -- (1,3) -- (3,3) -- (3,10) -- (4,10);
\foreach \x in {(0,0), (0,1), (0,2), (1,2), (1,3), (2,3), (3,3), (3,4), (3,5), (3,6), (3,7), (3,8), (3,9), (3,10), (4,10)}
	\fill \x circle (4pt);
\end{tikzpicture}
\caption{The lattice path $p(\sigma)$ corresponding to the preference order $\sigma = (1,1,-1,1,-1,-1,1,1,1,1,1,1,1,-1)$. The dotted lines $y=x$ and $y=x+7$ have been drawn to show that this path has drift $7$.}\label{fig:8 diners path}
\end{figure}
\end{example}

The problem of the clairvoyant ma\^{i}tre d' begins with a queue of diners possessing a particular preference order, which gets revealed to the ma\^{i}tre d' one diner at a time. The preference order is the only input to the problem. We can now give bounds for the number of unbalanced benches forced by a particular preference order. Note that we will write ``$-\sigma$'' to indicate changing the sign of each entry of $\sigma$.

\begin{proposition}[Bounds for unbalanced benches]\label{prop:unbalanced}
Fix a positive integer $n$, with $q := \lfloor n/3 \rfloor$, and write $n=3q+r$. Fix $\sigma \in \{\pm1\}^n$, with $h:=\max\{\h(\sigma),\h(-\sigma)\}$. For any nonnegative integer $i$, if $h \ge q + r + 2i - 1$, then any bench collection $\beta$ has at least $i$ unbalanced benches, and thus $b(\beta,\sigma)\leq q-i$. In particular, $\nu_{\max}(\sigma) \leq q-i$. 
\end{proposition}

\begin{proof}
The claim about $\nu_{\max}(\sigma)$ will follow from Observation~\ref{obs:balance} if we can establish the claimed lower bound on unbalanced benches.

The proposition is a tautology for $i=0$ (every bench collection has at least 0 unbalanced benches), so let us assume that $i$ is positive.

Suppose $h\geq q+r+2i-1$, and without loss of generality, suppose $h=\h(\sigma)$. Then for some $s\leq n$, we have
$$y_s - x_s = q + r + 2i - 1,$$
and thus
$$y_s = 2x_s + (q-x_s) + r + 2(i-1)+1,$$
where we have regrouped the terms to help conceptualize our pigeonhole-style argument. Notice also that we must have $x_s \leq q$, since otherwise $x_s\geq q+1,$ forcing $s=x_s+y_s\geq 3q+r+2i+1>n;$ but we know that $s\leq n$. In particular, $y_s \geq 2x_s$.

Consider how these first $s$ diners are distributed among the $q$ benches of a bench collection $\beta$. If we exhaust the supply of negative diners (those whose $\sigma$ values are $-1$) by placing each of the $x_s$ of them with two positive diners, this gives us $x_s$ benches that have the potential to be balanced. %are not necessarily unbalanced. 

This leaves $(q-x_s) + r+ 2(i-1)+1$ positive diners. Observe that, as we continue to distribute the remaining $s - 3x_s$ of these initial diners, any bench that acquires at least two positive diners will necessarily be unbalanced.

There are $q-x_s$ empty benches and $r$ remainder seats, so we can first place one positive diner per bench and in each of the remainder seats before we are forced to create an unbalanced bench. This leaves $2(i-1)+1$ more positive diners to distribute, among the $q-x_s$ benches that currently possess only one positive diner. To do this with the fewest unbalanced benches, we add two positive diners each to $i-1$ of these benches, and one more positive diner to another bench, creating a total of $i$ unbalanced benches. Any subsequent positioning of the remaining $n-s$ diners can do nothing to remove the $i$ unbalanced benches that were necessary and forced by these initial $s$ diners. 
\end{proof}

\section{Clairvoyant trap setting}\label{sec:clairvoyant}

We now introduce an optimal algorithm for solving the clairvoyant ma\^{i}tre d' problem.

To our previous language of balanced and unbalanced benches in a bench collection, we add the following. In a mild abuse of notation, we say a bench is:
\begin{itemize}
\item \textbf{open} if no diners have been assigned yet, and
\item \textbf{primed} if it has only two diners assigned so far (necessarily in the leftmost and rightmost seats).
\end{itemize} 
Note that previously we defined benches as triples of diners whereas now we also use ``bench'' to refer to the seats in which such diners sit. The usage will be clear from context. 

Remainder seats, when present, are also considered to be primed benches. The clairvoyant trap setting algorithm creates a bench seating arrangement as follows. As before, there is a sequence of $n$ diners with preference order $\sigma \in \{\pm1\}^n$, and we write $n=3q+r$, with $q=\lfloor n/3 \rfloor$. The goal of this algorithm is to put a negative (i.e., left-preferring) diner in the rightmost seat of each bench and a positive diner in the leftmost seat of each bench, for as long as possible. Whenever we cannot take such a step, we next try to fill a primed bench, hopefully creating a napkinless diner. If this, too, is impossible, then we are forced to create an unbalanced bench.

\medskip

\noindent\textsc{Algorithm $\alg$ (clairvoyant trap setting):}
We initialize $q$ empty benches, $B_1,\ldots, B_q$, and $r$ remainder seats. For each $j$, the diners in bench $B_j$ will be placed in Seats $\{3j-2,3j-1,3j\}$ at the table. For each Diner $i=1,\ldots,n$, the host divines their preference $\sigma_i$ and acts as follows. 
\begin{enumerate}
    \item \textbf{If $\bm{\sigma_i = +1}$} then
     \begin{enumerate}
      \item \textbf{if there exists a bench with an unassigned leftmost seat} then the host assigns Diner $i$ to the leftmost seat in the lowest-numbered such bench,
      \item \textbf{else if there exists a primed bench} then the host assigns Diner $i$ to the lowest-numbered such bench (this will be a center seat or remainder seat; make it the leftmost available remainder seat in the latter case),
      \item \textbf{else}  the host assigns Diner $i$ to the rightmost seat in the lowest-numbered bench possible (which necessarily becomes unbalanced). 
     \end{enumerate}
    \item \textbf{Else ($\bm{\sigma_i = -1}$)} 
     \begin{enumerate}
      \item \textbf{if there exists a bench with an unassigned rightmost seat} then the host assigns Diner $i$ to the rightmost seat in the lowest-numbered such bench,
      \item \textbf{else if there exists a primed bench}  then the host assigns Diner $i$ to the lowest-numbered such bench (this will be a center seat or remainder seat; make it the leftmost available remainder seat in the latter case),
      \item \textbf{else} the host assigns Diner $i$ to the leftmost seat in the lowest-numbered bench possible (which necessarily becomes unbalanced).
     \end{enumerate}
\end{enumerate}

Let $\beta_C(\sigma)$ denote the bench seating arrangement produced by Algorithm $\alg$ applied to $\sigma$, after first cycling the list to ensure that Diner $1$ sits in Seat $1$, as was done in Definition~\ref{defn:bench seating order}.

\begin{obs}\label{obs:to reach 1c and 2c in the algorithm}
To deploy step (1c) in Algorithm $\alg$, all leftmost seats have already been assigned, and all remaining unassigned seats are in benches that have both the rightmost and the center seats unassigned. To deploy step (2c) in Algorithm $\alg$, all rightmost seats have already been assigned, and all remaining unassigned seats are in benches that have both the leftmost and the center seats unassigned.  In either case, all remainder seats have already been filled as well. 
\end{obs}

We demonstrate Algorithm $\alg$ with our running example.

\begin{example}\label{ex:14 diners algorithm}
Consider, once again, the $14$ diners with preference order 
$$\sigma = (1,1,-1,1,-1,-1,1,1,1,1,1,1,1,-1).$$
Then Algorithm $\alg$ produces the bench seating arrangement
$$\left(1, \Circled{8}, -3 \Big| 2, \Circled{9}, -5 \Big| 4, \Circled{10}, -6 \Big| 7, -14, 13 \Big| 11, 12\right).$$
In this arrangement, Diner $14$ is frustrated, while Diners $8$, $9$, and $10$ (circled) are napkinless.  Thus $\nu(\beta_C(\sigma)) = 3$, and indeed it is impossible to create more napkinless diners from this preference order. 
\end{example}

We now prove some key features of Algorithm $\alg$. 

\begin{lemma}[Unbalanced steps of Algorithm $\alg$]\label{lem:unbalance}
Unbalanced benches are created by Algorithm $\alg$ in steps (1c) and (2c) only. Furthermore:
\begin{itemize}
\item It is impossible for Algorithm $\alg$ to reach both steps (1c) and (2c) for the same preference order $\sigma$. 
\item If step (1c) is used for the $i$th time with Diner $s$, then $y_s = x_s + q + r + 2(i-1)+1$. 
\item If step (2c) is used for the $i$th time with Diner $s$, then $x_s = y_s + q + r + 2(i-1)+1$.
\end{itemize}
\end{lemma}

\begin{proof}
That we cannot reach both (1c) and (2c) follows from Observation~\ref{obs:to reach 1c and 2c in the algorithm}. We now prove the claimed equation for (1c). The analogous statement for (2c) follows by symmetry. 

Suppose that the first time we deploy step (1c) is while processing Diner $s$.  For one thing, this means $\sigma_s = +1$ and so Diner $s$ is a positive diner. Moreover all leftmost bench seats are assigned. Now recall Observation~\ref{obs:to reach 1c and 2c in the algorithm}, and observe that we must be about to create a primed bench by assigning Diner $s$ to its rightmost seat. This means we have not used step (2b) yet, because step (2b) cannot occur until all rightmost seats are full. In particular there are no benches with two negative diners, and no negative diner can be in a remainder seat. Put another way, each bench and remainder seat has at least one positive diner (which accounts for $q+r$ of the positive diners), and there are $x_{s-1}$ benches containing one negative diner and two positive diners. Thus, $y_{s-1} = x_{s-1}+q+r$. The sign of $\sigma_s$ means that $x_s = x_{s-1}$ and $y_s = y_{s-1}+1$, and thus $y_s = x_s + q+r+1$, as desired.

Notice that in general, after reaching step (1c), we have some benches with an unassigned rightmost seat and an assigned leftmost seat, and one primed bench. Each new negative diner will follow step (2a) and join 
one of those benches with an unassigned rightmost seat and assigned leftmost seat (until we run out of such benches), whereas new positive diners will join a primed bench with step (1b), unless there are no primed benches and we are forced to apply step (1c) again.

Suppose we next use step (1c) with Diner $t$. Until we reached step (1c) again, every negative diner joined a bench with an unassigned rightmost seat, to make a primed bench. Say this occurred $a$ times before the next instance of step (1c). There was already one primed bench, so there must be $a+1$ corresponding instances of step (1b) before Diner $t$ is seated. Thus $x_{t-1} = x_s +a$ and $y_{t-1} = y_s + a+1$. As Diner $t$ is a positive diner, we have $x_t=x_{t-1}$ and $y_t = y_{t-1}+1$. Therefore $y_t = x_t + q + r + 3$. 

Continuing in this way, we see that for each new deployment of step (1c), the difference between the number of positive diners and the number of negative diners increases by two, and the result follows. 
\end{proof}

We will now prove that, for any preference order $\sigma$, the bench seating arrangement produced by Algorithm $\alg$ achieves $\nu_{\max}(\sigma)$ napkinless diners. 

\begin{proposition}[Unbalanced benches in Algorithm $\alg$]\label{prop:unbalanced2}
Fix a positive integer $n$, with $q = \lfloor n/3 \rfloor$, and write $n=3q+r$. Fix $\sigma \in \{\pm1\}^n$, with $h=\max\{\h(\sigma),\h(-\sigma)\}$. 
\begin{enumerate}
\item If $h\leq q+r$, then $\beta_C(\sigma)$ has no unbalanced benches, and therefore $\nu_{\max}(\sigma)= q$.
\item If $h \in\{ q+r+2i-1, q+r+2i\}$, then $\beta_C(\sigma)$ has exactly $i$ unbalanced benches, and therefore $\nu_{\max}(\sigma) \geq b(\beta_C(\sigma))= q-i = \lfloor (n-h)/2 \rfloor$.
\end{enumerate}
\end{proposition}

\begin{proof}
First suppose $h\leq q+r$. In this scenario, we have $|y_i - x_i | \leq q+r$ for all $i$. By Lemma~\ref{lem:unbalance}, this means we never create an unbalanced bench. Thus $b(\beta_C(\sigma)) = q \leq \nu_{\max}(\sigma) \leq q$. So $\nu_{\max}(\sigma) = q$ as claimed. 

Now suppose $h \in\{ q+r+2i-1, q+r+2i\}$ for some $i\geq 1$ and without loss of generality, suppose $h = \h(\sigma)$. By Lemma~\ref{lem:unbalance}, this means step (1c) of Algorithm $\alg$ was applied to $i$ diners (but not $i+1$ diners), and each of these instances created an unbalanced bench. Therefore by Observation~\ref{obs:balance}, we have $b(\beta_C(\sigma)) = q-i \leq \nu_{\max}(\sigma)$. 

If $h=q+r+2i=n-2q+2i$, then $n-h=2(q-i)$. On the other hand, if $h=q+r+2i-1 = n-2q+2i-1$, then $n-h=2(q-i)+1$. In either case, we have $q-i = \lfloor (n-h)/2\rfloor$.
\end{proof}

By combining Proposition~\ref{prop:unbalanced2} with Proposition~\ref{prop:unbalanced} and Observation~\ref{obs:balance}, we obtain the following corollary, which establishes part (3) of Theorem~\ref{thm:main}.

\begin{corollary}\label{cor:optimum}
Algorithm $\alg$ is optimal. That is, for any preference order $\sigma \in \{\pm1\}^n$, with $h=\max\{\h(\sigma),\h(-\sigma)\}$, we have
\[
\nu_{\max}(\sigma) = b(\beta_C(\sigma)) = \min\{ q, \lfloor (n-h)/2\rfloor \}.
\]
\end{corollary}

\section{Results for the distribution of maximal napkinless numbers}\label{sec:distribution}

Now that we have characterized the maximal napkinless number for each $\sigma$ in terms of $h(\sigma):= \max\{\h(\sigma), \h(-\sigma)\}$, we turn to computing the probabilities. Denote by $\Pr(n,k)$ the probability that a preference order $\sigma \in \{\pm1\}^n$ has $\nu_{\max}(\sigma) =k$. We suppose that every preference order is equally likely, so that
\[
 \pr(n,k) = \frac{p_{n,k}}{2^n},
\]
where
\[
 p_{n,k} := |\{ \sigma\in \{\pm1\}^n : \nu_{\max}(\sigma) = k\}|.
\]

A direct consequence of Corollary~\ref{cor:optimum} is the following.

\begin{corollary}
Fix $n\geq 1$ and let $q=\lfloor n/3\rfloor$. Then
\[
 p_{n,k} = \begin{cases}
  |\{ \sigma\in \{\pm1\}^n : \lfloor (n-h(\sigma))/2 \rfloor = k\}| & \mbox{ if $0\leq k < q$,} \\
  |\{ \sigma\in \{\pm1\}^n : \lfloor (n-h(\sigma))/2 \rfloor \geq q\}| & \mbox{ if $k=q$.}
 \end{cases}
\]
\end{corollary}

The main result of the next section will be to show precisely how to compute $p_{n,k}$, establishing Theorem~\ref{thm:main} part (1).

\subsection{Counting lattice paths by drift}\label{sec:drift}

Let $P_{n,h}$ denote the set of $\{N,E\}$ lattice paths of length $n$ and drift $\h(\sigma)=h$ beginning at $(0,0)$.   

\begin{theorem}\label{thm:main lattice paths}
For $n\geq h \geq 0$,
\[
 |P_{n,h}| = \binom{n}{\lfloor (n-h)/2 \rfloor}.
\]
\end{theorem}

We prove Theorem~\ref{thm:main lattice paths} with a bijective argument shortly, but for the moment assume its validity and suppose $h > q+r$. A path with $n=3q+r$ steps cannot be both above $y=x+q+r$ and to the right of $x=y+q+r$, so assume $h=\h(\sigma)>\h(-\sigma)$.  If $h>q+r = n-2q$, then  $\lfloor (n-h)/2 \rfloor < q$ and by Corollary~\ref{cor:optimum} we have $\nu_{\max}(\sigma) = \lfloor (n-h)/2 \rfloor$. Setting $k=\lfloor (n-h)/2 \rfloor$, we deduce $h\in \{ n-2k-1, n-2k\}$. Considering the cases where $\h(-\sigma) > q+r$ as well gives us exactly two more such paths for each $h$. This gives us the amazingly simple formula below:
\[
p_{n,k} = 2|P_{n,n-2k-1}|+2|P_{n,n-2k}| = 4\binom{n}{k}.
\]

For all other $\sigma$, $h(\sigma) \leq q+r$, so $\lfloor (n-h)/2 \rfloor \geq q$. In this case, Corollary~\ref{cor:optimum} implies $\nu_{\max}(\sigma) = q$. Thus we obtain the following corollary which proves Theorem~\ref{thm:main} part (1).

\begin{corollary}
For $n\geq 1$,
\[
p_{n,k}=
 \begin{cases}
  4\displaystyle\binom{n}{k} & \mbox{ if $0\leq k < q$,} \\
  2^n-4\displaystyle\sum_{i=0}^{q-1}\binom{n}{i} & \mbox{ if $k=q$.}
 \end{cases}
\]
\end{corollary}

\begin{proof}[Proof of Theorem~\ref{thm:main lattice paths}]
The technique here will be to exhibit a bijection between the set $P_{n,h}$ and the set of $\{N,E\}$ paths in a $\ell\times (h+k)$ grid, where $k=\lceil (n-h)/2 \rceil$ and $\ell = \lfloor (n-h)/2 \rfloor$. Let $L_{n,h}$ denote the set of north- and east-step paths in this grid. Since these paths have a total of $n=h+k+\ell$ steps, and $\ell$ of them are $E$, we have $|L_{n,h}| = \binom{n}{\ell}$. 

To see our bijection with $P_{n,h}$, we decorate the $\ell\times (h+k)$ grid as follows. First, we color all vertices on or below the line $y=x+h$ black. Vertices strictly above the this line are colored white. The edges now come in three types which we will refer to as \emph{black edges} (two black vertices), \emph{white edges} (two white vertices), and \emph{gray edges} (one black, one white). 
See Figure~\ref{fig:Lnh} for an illustration.

\begin{figure}
\[
\begin{array}{cc}
 \begin{tikzpicture}[scale=.5]
  \foreach \x in {0,...,8}
   {
   \draw (\x,0) -- (\x,10);
   }
  \foreach \y in {0,...,10}
   {
   \draw (0,\y) -- (8,\y);
   }
  %Dips
  \draw[line width=5,color=gray,opacity=.5, cap=round] (0,2)--(0,3)--(1,3)--(1,4)--(2,4)--(2,5)--(3,5)--(3,6)--(4,6)--(4,7)--(5,7)--(5,8)--(6,8)--(6,9)--(7,9)--(7,10)--(8,10);
  %White vertices
  \foreach \x in {0,...,7}
  {
  \foreach \y in {\x,...,7}
  {
  \draw (\x,\y+3) node[inner sep=2,circle,fill=white,draw=black] {};
  }
  }
  %Black vertices
  \foreach \x in {0,...,8}
  {
  \foreach \y in {-2,...,\x}
  {
  \draw (\x,\y+2) node[inner sep=2,circle,fill=black] {};
  }
  }
  \draw[dashed] (-1,1)--(9,11);
 \end{tikzpicture}
 &
 \begin{tikzpicture}[scale=.5]
  \foreach \x in {0,...,6}
   {
   \draw (\x,0) -- (\x,12);
   }
  \foreach \y in {0,...,12}
   {
   \draw (0,\y) -- (6,\y);
   }
  %Dips
  \draw[line width=5,color=gray,opacity=.5, cap=round] (0,5)--(0,6)--(1,6)--(1,7)--(2,7)--(2,8)--(3,8)--(3,9)--(4,9)--(4,10)--(5,10)--(5,11)--(6,11)--(6,12);
  %White vertices
  \foreach \x in {0,...,6}
  {
  \foreach \y in {\x,...,6}
  {
  \draw (\x,\y+6) node[inner sep=2,circle,fill=white,draw=black] {};
  }
  }
  %Black vertices
  \foreach \x in {0,...,6}
  {
  \foreach \y in {-5,...,\x}
  {
  \draw (\x,\y+5) node[inner sep=2,circle,fill=black] {};
  }
  }
  \draw[dashed] (-1,4)--(7,12);
 \end{tikzpicture}
 \\
 (a) & (b)
 \end{array}
\]
\caption{The $\{N,E\}$ lattice path grid for set $L_{n,h}$ with (a) $n=18$, $h=2$, and (b) $n=18$, $h=5$.}\label{fig:Lnh}
\end{figure}
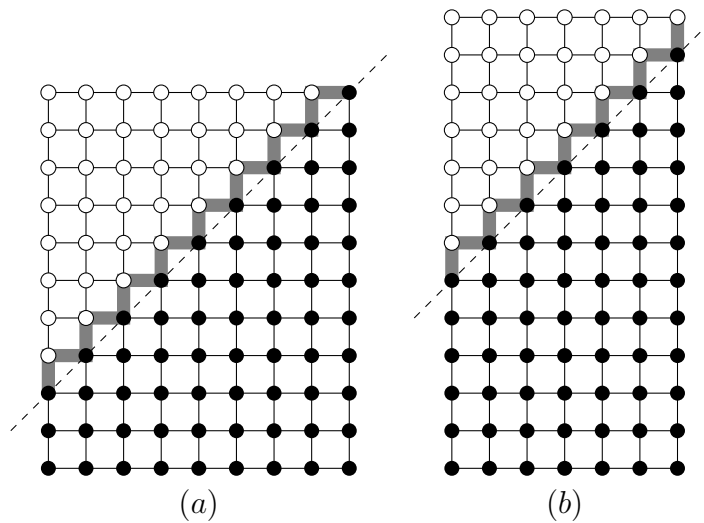

Let $p \in P_{n,h}$. We establish some terminology related to such a path. While a path of drift $h$ may reach that height a number of times, such a path has a rightmost point at which it achieves height $h$. We call this vertex the \emph{zenith} of the path. An east step $E$ is called a \emph{dip} if the height of the path beyond this step never exceeds the height at this step.

We will now color the vertices and edges of $p$ in a way that makes it easier to explain our mapping. First of all, the vertices up to and including the zenith will be colored black (and hence all those edges are black edges). Next, each dip will be colored gray, with one white vertex and one black vertex. (The alternating pattern begins black-white, starting at the zenith.) Since each dip must have one black vertex and one white vertex, any non-dip edges beyond the zenith will occur between consecutive dips, and they are singly-colored, with color determined by the color of the vertices of the nearest dips. 

Here is an example of an $\{N,E\}$ path of length $18$ and height $2$: 
\[
p=NNEENENNEEENEEEENE.
\]
It is drawn in the plane as shown in Figure~\ref{fig:ex}. We can see the zenith occurs after the eighth step, and there are a total of six dips.

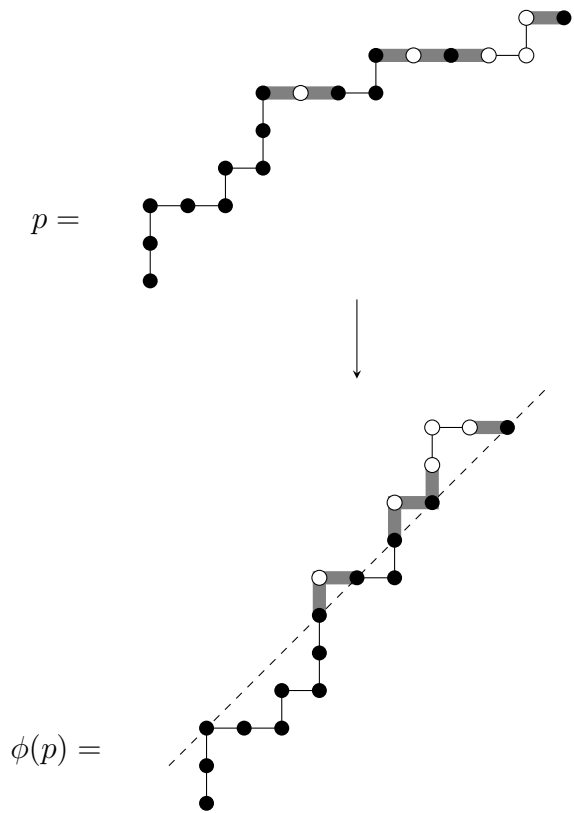
\begin{figure}
\[
\begin{tikzpicture}[>=stealth]
\draw (-4,5) node {$p=$};
\draw (-4,-2) node {$\phi(p)=$};
\draw (0,6) node (p) {
 \begin{tikzpicture}[scale=.5]
  %Basic path
  \draw (0,0) node[inner sep=2,circle,fill=black] {}--(0,1) node[inner sep=2,circle,fill=black] {}--(0,2) node[inner sep=2,circle,fill=black] {}--(1,2) node[inner sep=2,circle,fill=black] {}--(2,2) node[inner sep=2,circle,fill=black] {}--(2,3) node[inner sep=2,circle,fill=black] {}--(3,3) node[inner sep=2,circle,fill=black] {}--(3,4) node[inner sep=2,circle,fill=black] {}--(3,5) --(6,5) -- (6,6) -- (10,6) -- (10,7) -- (11,7);
  %Dips
  \draw[line width=5,color=gray,opacity=.5, cap=round] (3,5)--(5,5);
  \draw[line width=5,color=gray,opacity=.5, cap=round] (6,6)--(9,6);
  \draw[line width=5,color=gray,opacity=.5, cap=round] (10,7)--(11,7);
  %White vertices
  \draw (4,5) node[inner sep=2,circle,fill=white,draw=black] {};
  \draw (7,6) node[inner sep=2,circle,fill=white,draw=black] {};
  \draw (9,6) node[inner sep=2,circle,fill=white,draw=black] {};
  \draw (10,6) node[inner sep=2,circle,fill=white,draw=black] {};
  \draw (10,7) node[inner sep=2,circle,fill=white,draw=black] {};
  %Black vertices
  \draw (3,5) node[inner sep=2,circle,fill=black] {};
  \draw (5,5) node[inner sep=2,circle,fill=black] {};
  \draw (6,5) node[inner sep=2,circle,fill=black] {};
  \draw (6,6) node[inner sep=2,circle,fill=black] {};
  \draw (8,6) node[inner sep=2,circle,fill=black] {};
  \draw (11,7) node[inner sep=2,circle,fill=black] {};
 \end{tikzpicture}
 };
 \draw (0,0) node (ph) {
 \begin{tikzpicture}[scale=.5]
  %Basic path
  \draw (0,0) node[inner sep=2,circle,fill=black] {}--(0,1) node[inner sep=2,circle,fill=black] {}--(0,2) node[inner sep=2,circle,fill=black] {}--(1,2) node[inner sep=2,circle,fill=black] {}--(2,2) node[inner sep=2,circle,fill=black] {}--(2,3) node[inner sep=2,circle,fill=black] {}--(3,3) node[inner sep=2,circle,fill=black] {}--(3,4) node[inner sep=2,circle,fill=black] {}--(3,6) --(5,6) -- (5,8) -- (6,8) -- (6,10) -- (7,10) -- (8,10);
  %Dips
  \draw[line width=5,color=gray,opacity=.5, cap=round] (3,5)--(3,6)--(4,6);
  \draw[line width=5,color=gray,opacity=.5, cap=round] (5,7)--(5,8)--(6,8)--(6,9);
  \draw[line width=5,color=gray,opacity=.5, cap=round] (7,10)--(8,10);
  %White vertices
  \draw (3,6) node[inner sep=2,circle,fill=white,draw=black] {};
  \draw (5,8) node[inner sep=2,circle,fill=white,draw=black] {};
  \draw (6,9) node[inner sep=2,circle,fill=white,draw=black] {};
  \draw (6,10) node[inner sep=2,circle,fill=white,draw=black] {};
  \draw (7,10) node[inner sep=2,circle,fill=white,draw=black] {};
  %Black vertices
  \draw (3,5) node[inner sep=2,circle,fill=black] {};
  \draw (4,6) node[inner sep=2,circle,fill=black] {};
  \draw (5,6) node[inner sep=2,circle,fill=black] {};
  \draw (5,7) node[inner sep=2,circle,fill=black] {};
  \draw (6,8) node[inner sep=2,circle,fill=black] {};
  \draw (8,10) node[inner sep=2,circle,fill=black] {};
  \draw[dashed] (-1,1)--(9,11);
 \end{tikzpicture}
  };
 \draw[->] (p) -- (ph);
\end{tikzpicture}
\]
\caption{An example of the map $p \to \phi(p)$, for a path with parameters $n=18$ and $h=2$. }\label{fig:ex}
\end{figure}

We now present the bijection $\phi: P_{n,h} \to L_{n,h}$. The map takes edges to edges as indicated in Figure~\ref{fig:phi}, fixing the type of step if it ends in a black vertex, and reversing the orientation of the step if it ends in a white vertex. In Figure~\ref{fig:ex}, we see our example path $p$ transformed into 
\[
\phi(p) = NNEENENNNEENNENNEE.
\]
The edgewise definition of $\phi$ shows that the map is one-to-one and invertible. The trickier thing to see is that map actually takes paths of drift $h$ into the set $L_{n,h}$.

\begin{figure}
\[
\begin{array}{|c || c| c |c |c |c |c|}
\hline
&&&&&& \\
s &
\begin{tikzpicture}
\draw (0,0) node[inner sep=2,circle,fill=black] {}--(1,0) node[inner sep=2,circle,fill=black] {};
\end{tikzpicture}
&
\begin{tikzpicture}[baseline=0.5cm]
\draw[draw=none] (-.5,0) node[inner sep=2,circle] {} -- (.5,0) node[inner sep=2,circle] {};
\draw (0,0) node[inner sep=2,circle,fill=black] {}--(0,1) node[inner sep=2,circle,fill=black] {};
\end{tikzpicture}
& 
\begin{tikzpicture}
\draw (0,0) node[inner sep=2,circle,fill=white,draw=black] {}--(1,0) node[inner sep=2,circle,fill=white,draw=black] {};
\end{tikzpicture}
&
\begin{tikzpicture}[baseline=0.5cm]
\draw (0,0) node[inner sep=2,circle,fill=white,draw=black] {}--(0,1) node[inner sep=2,circle,fill=white,draw=black] {};
\end{tikzpicture}
&
\begin{tikzpicture}
\draw[line width=5,color=gray,opacity=.5, cap=round] (0,0)--(1,0);
\draw (0,0) node[inner sep=2,circle,fill=black] {}--(1,0) node[inner sep=2,circle,fill=white,draw=black] {};
\end{tikzpicture}
&
\begin{tikzpicture}
\draw[line width=5,color=gray,opacity=.5, cap=round] (0,0)--(1,0);
\draw (0,0) node[inner sep=2,circle,fill=white,draw=black] {}--(1,0) node[inner sep=2,circle,fill=black] {};
\end{tikzpicture}
 \\
& &&&&&\\
\hline
& &&&&&\\
\phi(s) 
&
\begin{tikzpicture}
\draw (0,0) node[inner sep=2,circle,fill=black] {}--(1,0) node[inner sep=2,circle,fill=black] {};
\end{tikzpicture}
& 
\begin{tikzpicture}[baseline=0.5cm]
\draw (0,0) node[inner sep=2,circle,fill=black] {}--(0,1) node[inner sep=2,circle,fill=black] {};
\end{tikzpicture}
&
\begin{tikzpicture}[baseline=0.5cm]
\draw (0,0) node[inner sep=2,circle,fill=white,draw=black] {}--(0,1) node[inner sep=2,circle,fill=white,draw=black] {};
\end{tikzpicture}
&
\begin{tikzpicture}
\draw (0,0) node[inner sep=2,circle,fill=white,draw=black] {}--(1,0) node[inner sep=2,circle,fill=white,draw=black] {};
\end{tikzpicture}
&
\begin{tikzpicture}[baseline=0.5cm]
\draw[line width=5,color=gray,opacity=.5, cap=round] (0,0)--(0,1);
\draw (0,0) node[inner sep=2,circle,fill=black] {}--(0,1) node[inner sep=2,circle,fill=white,draw=black] {};
\end{tikzpicture}
&
\begin{tikzpicture}
\draw[line width=5,color=gray,opacity=.5, cap=round] (0,0)--(1,0);
\draw (0,0) node[inner sep=2,circle,fill=white,draw=black] {}--(1,0) node[inner sep=2,circle,fill=black] {};
\end{tikzpicture}
\\
&&&&&&\\
\hline
\end{array}
\]
\caption{Definition of the map $\phi$, depicted in terms of the image of each possible edge in a path.}\label{fig:phi}
\end{figure}

Let $\ell=\lfloor (n-h)/2 \rfloor$, i.e., the width of the grid for paths in $L_{n,h}$. Let $p \in P_{n,h}$. We need to show that exactly $\ell$ steps of $p$ are mapped to $E$ steps in $\phi(p)$. That is, we claim $b+w+g=\ell$, where $b$ is the number of black $E$ steps, $w$ is the number of white $N$ steps, and $g$ is the number of dips that begin with a white vertex.

To reach this conclusion, we first count all types of steps. Let $b(E)$ denote the number of black edges that are $E$ steps, and let $b(N)$ denote the number of black edges that are $N$ steps. Similarly define $w(E)$ and $w(N)$ for the white steps. Since all dips are $E$ steps, we let $g(N)$ denote the number of dips that have their black vertex on the left; $g(E)$ denotes the number of dips that have the white vertex on the left.

For any path $p \in P_{n,h}$, we have: 
\begin{align*}
b(N) &= b(E) + h,\\
w(N) &= w(E),\\
g(N) &= \begin{cases} g(E) & \mbox{ if $n-h$ is even},\\
  g(E)+1 & \mbox{ if $n-h$ is odd.}
  \end{cases}
\end{align*}
Therefore,
\begin{align*}
n&=b(N)+b(E) +w(N)+w(E)+ g(N)+g(E),\\
 &= 2b(E)+h + 2w(N) + 2g(N)+\begin{cases} 0 & \mbox{ if $n-h$ is even},\\
  1 & \mbox{ if $n-h$ is odd,}
  \end{cases}
\end{align*}
and so $\ell=\lfloor (n-h)/2 \rfloor = b(E) + w(N) + g(N)$, as claimed.  
\end{proof}

\begin{remark}[Other proofs of Theorem~\ref{thm:main lattice paths}]
In entry A061554 of \cite{oeis}, there is a comment of Gerald McGarvey that Theorem~\ref{thm:main lattice paths} can be proved using the recurrences $|P_{n,0}|=|P_{n-1,0}|+|P_{n-1,1}|$ and 
\[
 |P_{n,h}| = |P_{n-1,h-1}|+|P_{n-1,h+1}|,
\]
for $h\geq 1$. We leave details to the reader, but the idea is to prepend a step to a path of length $n-1$. If the new step is $N$, then the drift increases by one. If the new step is $E$, then the drift decreases by one (or remains at zero).

Ira Gessel (private communication) has provided a third argument, whereby $\{N,E\}$ paths are interpreted as parenthesizations with $N\to ``)"$ and $E\to ``("$. In this formulation, the drift of a path is the number of unpaired right parentheses, while $y_n-x_n$ is the number of unpaired right parentheses minus the number of unpaired left parentheses. We say the ``ending height'' of a path is $\max\{0,y_n-x_n\}$. For example, the path $p$ in Figure~\ref{fig:ex} becomes
\[
  \rb\rb(()())\lb\lb()\lb\lb\lb()\lb,
\]
where we have enlarged the unmatched parentheses for emphasis. The path has drift 2 and ending height $\max\{0, 2-6\}=0$. 

Notice that all the unpaired right parentheses are necessarily to the left of the unpaired left parentheses. If $p$ has drift $k$, we can form a bijection $p\to \psi(p)$ that amounts to sequentially converting the unpaired left parentheses into unpaired right parentheses, until we have a path with ending height $k$. In general this bijection is different from that used in the proof above. Continuing our example,
\[
\begin{tikzpicture}[>=stealth,yscale=.75]
\draw (0,0) node (a) { \rb\rb(()())\lb\lb()\lb\lb\lb()\lb };
\draw (0,-2) node (b) { \rb\rb(()())\rb\lb()\lb\lb\lb()\lb };
\draw (0,-4) node (c) { \rb\rb(()())\rb\rb()\lb\lb\lb()\lb };
\draw (0,-6) node (d) { \rb\rb(()())\rb\rb()\rb\lb\lb()\lb };
\draw[->] (a)--(b);
\draw[->] (b)--(c);
\draw[->] (c)--(d);
\draw (-2,0) node[left] {$p=$};
\draw (2.5,0) node[right] {$\h(p)=2$};
\draw (-2,-6) node[left] {$\psi(p)=$};
\draw (2.5,-6) node[right] {$\max\{0,y_n-x_n\} = 2$};
\end{tikzpicture}
\]
Note that $\psi(p)\neq \phi(p)$ in this example. We leave further details of this bijection to the reader.
\end{remark}

\subsection{Expected number of napkinless}

The main result of this section is to show that as the table gets large, the clairvoyant ma\^{i}tre d' expects (using Algorithm $\alg$) to get very close to $1/3$ of the diners napkinless on average. The following result will establish part (2) of Theorem~\ref{thm:main}, completing its proof.

\begin{proposition}
For each $n\geq 3$, with $q=\lfloor n/3\rfloor$, the expected number of napkinless diners using Algorithm $\alg$ is:
\[
 q - \frac{1}{2^{n-2}}\sum_{k=0}^{q-1} (q-k)\binom{n}{k}.
\]
In particular, as $n\to \infty$ the expectation converges to $\lfloor n/3 \rfloor$.
\end{proposition}

\begin{proof}
Letting $E_n$ denote the expectation, we have 
\begin{align*}
 E_n &= \sum_{k=0}^q k\pr(n,k),\\
  &= \frac{1}{2^n}\sum_{k=0}^{q-1} 4k\binom{n}{k} + \frac{q}{2^n}\left[ 2^n - 4\sum_{k=0}^{q-1}\binom{n}{k}\right],\\
  &= q - \frac{4}{2^n} \sum_{k=0}^{q-1} (q-k)\binom{n}{k}.
\end{align*}
To show the claimed convergence, we use the very coarse bound of
\[
4\sum_{k=0}^{q-1} (q-k)\binom{n}{k} \leq 4\sum_{k=0}^{q-1} q\binom{n}{q} \leq 4q^2\binom{n}{q},
\]
and show $4q^2\binom{n}{q}/2^n \to 0$.

Using Stirling's approximation, $n!\approx \sqrt{2\pi n}(n/e)^n$, with $q\approx n/3$, we have 
\begin{align*}
 4q^2\binom{n}{q} &= 4q^2\frac{ n!}{q!(n-q)!}\\
 &\approx Cn^2\frac{ \sqrt{n}(n/e)^n}{\sqrt{n/3}(n/3e)^{n/3}\sqrt{2n/3}(2n/3e)^{2n/3}}\\
 &= C_1 \frac{n^{3/2}3^n}{(2^n)^{2/3}},
\end{align*}
for some constant $C_1$.
Dividing by $2^n$, we have
\[
4q^2\binom{n}{q}/2^n \approx C_1 \frac{\sqrt{n}3^n}{(2^{5/3})^n}.
\]
As $2^{5/3} \approx 3.17$, this expression, which is approximately $q-E_n$, converges to zero. 
\end{proof}

\section{Further questions}

There are many directions for further study, and we suggest two here.

\subsection{Uneven napkin preferences}

In \cite{CP, Eriksen, Sudbury}, the authors studied the distribution of napkinless diners (for the random seating order case) with the assumption that a diner reaches left with some fixed probability $p$ and right with probability $1-p$. As we have only considered the case of $p=1/2$, it would be interesting to similarly generalize the results here. (Likewise, \cite{APST} only considers $p=1/2$.) As $p\to 1$, the ma\^{i}tre d' has fewer opportunities to trap diners, so we know that $E_n/n \to 0$ as the diners are more aligned in their napkin preferences, but what exactly is the dependence on $p$? 

A closely related question is to compute the distribution of drift when the paths are distributed with an arbitrary binomial distribution. One interesting wrinkle is that two paths of length $n$ can have the same drift but a different number of $E$ steps, and hence a different probability. For example, paths $NNNENE$ and $NNNEEE$ both have drift 3. But if an $N$ step occurs with probability $p$, these two paths occur with probability $p^4(1-p)^2$ and $p^3(1-p)^3$, respectively.

\subsection{Graph-theoretic generalization}

%\bridget{I think this section is really describing non-(circular tables), not non-circular tables. Is that right? I was a little confused when I initially read the section, but perhaps it is instead describing a more general situation of a picnic blanket with napkins scattered around and people can sit down and grab whichever napkin they want. Is that right?} 
The malicious ma\^{i}tre d' problem can be viewed as a graph theory problem played on a cycle graph with $2n$ vertices. If we color the vertices black (for diners) and white (for napkins) in alternating fashion, then each diner-napkin pair corresponds to an edge between a black vertex and a white vertex. A seating order results in a matching of the graph in which we label the black vertices with the diner number. Napkinless diners correspond to isolated vertices in the matching. The example in Figure \ref{fig:circle} would then be the following matching:
\[
\begin{tikzpicture}[scale=1,baseline=0]
x\node[draw=none,minimum size=4cm,regular polygon,regular polygon sides=8] (a) {};
\node[draw=none,minimum size=4cm,regular polygon,regular polygon sides=8] (c) {};
\foreach \x in {1,2,3,4,5,6,7,8}
\draw[thick] (a.side 1) -- (c.corner 2);
\draw[thick] (a.side 3) -- (c.corner 3);
\draw[thick] (a.side 4) -- (c.corner 4);
\draw[thick] (a.side 5) -- (c.corner 6);
\draw[thick] (a.side 6) -- (c.corner 7);
\draw[thick] (a.side 8) -- (c.corner 8);
\foreach \x in {1,2,3,4,5,6,7,8}
  {\fill[white] (c.corner \x) circle (3pt);
  \draw (c.corner \x) circle (3pt);
  }
\draw (a.side 1) node[circle,fill=black,inner sep=2] {} node[above] {$1$};
\draw (a.side 2) node[circle,fill=black,inner sep=2] {} node[above left] {$5$};
\draw (a.side 3) node[circle,fill=black,inner sep=2] {} node[left] {$2$};
\draw (a.side 4) node[circle,fill=black,inner sep=2] {} node[below left] {$8$};
\draw (a.side 5) node[circle,fill=black,inner sep=2] {} node[below] {$4$};
\draw (a.side 6) node[circle,fill=black,inner sep=2] {} node[below right] {$6$};
\draw (a.side 7) node[circle,fill=black,inner sep=2] {} node[right] {$7$};
\draw (a.side 8) node[circle,fill=black,inner sep=2] {} node[above right] {$3$};
\end{tikzpicture}
\]

With this framework, we can consider other bipartite graphs and suppose the ma\^{i}tre d' seats diners at black vertices one at a time. When seated, a diner is equally likely to select the napkin at any one of its neighboring white vertices, if available. How many napkinless diners (isolated black vertices) do we expect?

For example, suppose the graph below has napkins at white vertices and diners sit at black vertices:
\[
\begin{tikzpicture}
\draw (0,0)--(0,1)--(2,1)--(2,0)--(0,0);
\draw (1,0)--(1,1);
    \draw (0,0) node[inner sep=2,circle,fill=white,draw=black] {};
    \draw (1,1) node[inner sep=2,circle,fill=white,draw=black] {};
    \draw (2,0) node[inner sep=2,circle,fill=white,draw=black] {};
    \draw (1,0) node[inner sep=2,circle,fill=black] {};
    \draw (0,1) node[inner sep=2,circle,fill=black] {};
    \draw (2,1) node[inner sep=2,circle,fill=black] {};
\end{tikzpicture}
\]
Here are two possible seating outcomes:
\[
\begin{tikzpicture}[baseline=0]
\draw (0,0)--(0,1);
\draw (1,0)--(2,0);
\draw (1,1)--(2,1);
    \draw (0,0) node[inner sep=2,circle,fill=white,draw=black] {};
    \draw (1,1) node[inner sep=2,circle,fill=white,draw=black] {};
    \draw (2,0) node[inner sep=2,circle,fill=white,draw=black] {};
    \draw (1,0) node[inner sep=2,circle,fill=black] {};
    \draw (0,1) node[inner sep=2,circle,fill=black] {};
    \draw (2,1) node[inner sep=2,circle,fill=black] {};
    \draw (1,0) node[below] {$1$};
    \draw (0,1) node[above] {$2$};
    \draw (2,1) node[above] {$3$};
\end{tikzpicture}
\qquad \mbox{ and } \qquad
\begin{tikzpicture}[baseline=0]
\draw (0,0)--(1,0);
\draw (1,1)--(2,1);
    \draw (0,0) node[inner sep=2,circle,fill=white,draw=black] {};
    \draw (1,1) node[inner sep=2,circle,fill=white,draw=black] {};
    \draw (2,0) node[inner sep=2,circle,fill=white,draw=black] {};
    \draw (1,0) node[inner sep=2,circle,fill=black] {};
    \draw (0,1) node[inner sep=2,circle,fill=black] {};
    \draw (2,1) node[inner sep=2,circle,fill=black] {};
    \draw (1,0) node[below] {$1$};
    \draw (0,1) node[above] {$3$};
    \draw (2,1) node[above] {$2$};
\end{tikzpicture}
\]
The numbers indicate the order in which the diners were seated. We can see on the left that each diner receives a napkin, while on the right, diner 3 is napkinless.

Of course this question can be asked of various families of graphs, under a variety of assumptions about the ma\^{i}tre d' and the distribution of diner preferences. The case of path graphs is implicitly dealt with in  studying ``straight tables'' in \cite{APST, CP}.

\bigskip

\noindent\textbf{Acknowledgments.}
Work of Shirman was partially supported by a Research Project Assistantship from the Department of Mathematical Sciences at DePaul University. Work of Tenner was partially supported by NSF Grant DMS-2054436.

\end{document}